\documentclass{article}
\usepackage{amsmath}
\usepackage{amssymb}
\usepackage{amsthm}
\usepackage{amsfonts}
\usepackage{diagbox}
\newtheorem{theorem}{Theorem}[section]

\newtheorem{lemma}[theorem]{Lemma}
\newtheorem{corollary}[theorem]{Corollary}

\newtheorem{definition}[theorem]{Definition}

\title{Revisiting the universal linear algebraic model for the characteristic two case}
\author{M\'at\'e L. Juh\'asz\\[2mm]
Eszterh\'azy K\'aroly University\\
\texttt{juhasz.mate@uni-eszterhazy.hu}
}

\DeclareMathOperator{\Ker}{Ker}

\DeclareMathOperator{\Iso}{Iso}

\DeclareMathOperator{\Ort}{Ort}

\DeclareMathOperator{\Rad}{Rad}

\DeclareMathOperator{\Arf}{Arf}

\DeclareMathOperator{\cotr}{cotr}

\begin{document}

\maketitle

\begin{abstract}
In \cite{juhasz2016univ}, a universal linear algebraic model was proposed for describing homogeneous conformal geometries, such as the spherical, Euclidean, hyperbolic, Minkowski, anti-de Sitter and Galilei planes (\cite{yaglom1969}). This formalism was independent from the underlying field, providing an extension and general approach to other fields, such as finite fields. Some steps were taken even for the characteristic $2$ case.

In this article, we undertake the study of the characteristic $2$ case in more detail. In particular, the concept of virtual quadratic spaces is used (\cite{juhasz2016isom}), and a similar result is achieved for finite fields of characteristic $2$ as for other fields. Some differences from the non-characteristic $2$ case are also pointed out.
\end{abstract}

\section{Introduction}

In \cite{juhasz2016univ}, a universal linear algebraic model was proposed for describing homogeneous conformal geometries.
In particular, it provided a uniform description of spherical, Euclidean, hyperbolic geometries, as well as Minkowski's, anti-de Sitter and the Galilei plane (\cite{yaglom1969}), using quadratic forms on vector spaces.
The model also gave a generalization of pseudo-Riemannian manifolds for arbitrary fields.
As is usual with quadratic forms, the characteristic $2$ case was mostly excluded from the article.

\section{Preliminaries}

Let us fix some field and denote it by $\mathbb{K}$.

\begin{definition}
A \textbf{quadratic space} is a pair $(V,Q)$ consisting of a vector space and a quadratic form on it. It has an \textbf{associated bilinear from} defined as $B(x,y)=Q(x+y)-Q(x)-Q(y)$. The \textbf{radical} of a quadratic form $Q$ is the set of vectors $v$ such that $Q(v+u)=Q(u)$, and it is denoted by $\Rad Q$. The \textbf{direct sum} of two quadratic spaces $(V,Q)$ and $(V',Q')$ is $(V\oplus V',Q\oplus Q')$ such that $(Q\oplus Q')(v\oplus v')=Q(v)+Q(v')$. An \textbf{isometry} is a linear map $\varphi\colon V\to V$ such that $Q(\varphi(v))=Q(v)$ for all $v\in V$. The \textbf{group of isometries} is denoted by $\Iso(V)$.
\end{definition}

\begin{definition}
Let us denote the function $u\to B(v,u)$ by $v^*$. A bilinear form is \textbf{non-degenerate} or \textbf{regular} if for every non-zero vector $v\in V$, the linear function $v^*$ is non-zero. A quadratic form is \textbf{non-degenerate} if its associated bilinear form is non-degenerate.
\end{definition}

\begin{definition}
A \textbf{hyperbolic space} of dimension $2$ is a pair $(\Sigma,B)$ is such that the bilinear form $B$ takes the form $B(u,v)=u_1v_2+u_2v_1$ in some basis. A hyperbolic space of dimension $2k$ is the orthogonal direct sum of $k$ hyperbolic spaces of dimension $2$ each.
\end{definition}

\begin{theorem}
Given two subspaces $U$, $V\le W$ of the same dimension such that neither of them contains degenerate vectors, any isometry between $U$ and $V$ extends into an isometry of $W$. (See \cite{kitaoka1993}, \textbf{Corollary 1.2.1.})
\end{theorem}

Let us turn towards the characteristic $2$ case. First, let us recall the definition of a perfect fields.

\begin{definition}
A field $\mathbb{F}$ is \textbf{perfect} if for any element in any finite field extension of $\mathbb{F}$, it is a simple root of its minimal polynomial.
\end{definition}

\begin{theorem}
Every finite field is perfect.
\end{theorem}

We will not need all the properties of perfect fields, only the following:

\begin{lemma}
If the characteristic of a field is $p$, polynomials of the form $x^p-a$ have a single solution.

In a perfect field $\mathbb{K}$ of characteristic $2$, the equation $x^2=a$ has a unique solution in $x$ for any $a\in\mathbb{K}$. Alternatively, the map $x\to x^2$ is a bijection and an automorphism of the additive group $\mathbb{K}^+$.
\end{lemma}

Over a perfect field of characteristic $2$, polynomials of the form $x^2-a$ are always reducible, and the irreducible quadratic polynomials will take the form $x^2+x+a$ instead.

The main issue with the characteristic $2$ case is that in some dimensions there is no non-degenerate quadratic form. For example, for $\mathbb{K}=\mathbb{F}_2$, there are non-degenerate quadratic forms only in even dimensions. However, by embedding an odd dimensional vector space inside an even dimensional quadratic space, we may salvage certain properties of quadratic spaces in the odd dimensional case as well.

Let us recall virtual quadratic spaces from \cite{juhasz2016isom} and a few related definitions.

\begin{definition}
A \textbf{virtual quadratic space} is a tuple $(V,Q,U)$ or $(V,U)$ for short, with $U$ a subspace of a vector space $V$ and $Q$ a non-degenerate quadratic form on $V$. Its \textbf{dimension} is $\dim U$. An \textbf{isometry} of the virtual quadratic space is an isometry of $(V,Q)$ that fixes $U^\perp$. The group of isometries is denoted by $\Iso(V,U)$.

Two virtual quadratic spaces $(V,Q,U)$ and $(V',Q',U')$ are \textbf{isomorphic} if $(U,Q|_U)$ and $(U',Q|_{U'})$ are isomorphic quadratic spaces.

A virtual quadratic space $(V,U)$ is \textbf{non-degenerate} if $\Rad Q=\{0\}$.
\end{definition}

The main result was the following:

\begin{theorem}\label{Tmain}
Consider a quadratic space $(Q_U,U)$. Then it can be embedded into a virtual quadratic space $(V,Q_V,U)$, and for such an embedding, $\Iso(V,U)$ depends only on $Q_U$. In fact, such a $V$ can always be chosen so that $\dim V=\dim U+\dim(U\cap U^\perp)$, even as a subspace of some other virtual quadratic space $(V',Q_{V'},U)$, which is equivalent to the condition that $U^\perp\subseteq U$ in $V$. Furthermore, if $\Rad Q_U=\{0\}$, the restriction map $\Iso(V,U)\to\Iso(U)$ is a surjective map.
\end{theorem}

\begin{lemma}
\label{Tvirtual2}
In a perfect field $\mathbb{K}$ of characteristic $2$, any non-degenerate virtual quadratic space $(V,U)$ is such that $\dim(U\cap U^\perp)\le 1$.
\end{lemma}

\begin{corollary}
In a perfect field $\mathbb{K}$ of characteristic $2$, a non-degenerate virtual quadratic space is always of the form $(V,V)$ if $2\mid\dim V$ or $(V,\Omega^\perp)$ for some vector $\Omega\in V$ if $2\nmid\dim V$.
\end{corollary}

A few further theorems are necesssary, which can be found in any reference book (see for instance \cite{kitaoka1993}).

\begin{definition}
Given a non-degenerate quadratic space $(V,Q)$ of dimension $2$ over a field of characteristic $2$ with basis $e_1$, $e_2$, the \textbf{Arf invariant} of $V$ is defined as $\Arf(V):={Q(e_1)Q(e_2)\over B(e_1,e_2)^2}$, and if $B(e_1,e_2)=0$, it is defined as $\infty$. Given a non-degenerate quadratic space $(V,Q)$ over a field of characteristic $2$ such that it decomposes as an orthogonal sum into $V_1\oplus\cdots\oplus V_k$ with a corresponding basis $e^1_1$, $e^1_2$, \dots, $e^k_1$, $e^k_2$, the Arf invariant is defined as the sum $\Arf(V)=\Arf(V_1)+\cdots+\Arf(V_k)$.
\end{definition}

\begin{lemma}
Let us denote by $\mathfrak{H}$ the operation $\mathfrak{H}(x)=x+x^2$, which is an \textbf{additive} function over a field $\mathbb{K}$ of characteristic $2$. Let us denote by $\mathfrak{H}(\mathbb{K})$ the image of $\mathfrak{H}$ over $\mathbb{K}$.

Given a non-degenerate quadratic space $(V,Q)$, its Arf invariant is independent from choice of basis, up to an additive term in $\mathfrak{H}(\mathbb{K})$. Alternatively, the Arf invariant of a non-degenerate quadratic space is well defined as an element of the additive quotient group $\mathbb{K}^+/\mathfrak{H}(\mathbb{K})$, unless it is $\infty$.
\end{lemma}

This is the analogue in characteristic $2$ of the lemma that the discriminant of a quadratic space is well defined as an element of the multiplicative quotient group $\mathbb{K}^\times/(\mathbb{K}^\times)^2$.

When the underlying field is perfect, the Arf invariant completely describes the quadratic space.

\begin{theorem}
Two non-degenerate quadratic spaces $(V_1,Q_1)$ and $(V_2,Q_2)$ of equal dimensions are isomorphic if and only if their Arf invariants differ in an element of $\mathfrak{H}(\mathbb{K})$, or both are equal to $\infty$.
\end{theorem}

Before commencing with our geometric discussion, one more algebraic property is needed:

\begin{lemma}
\label{Lindex}
Over a finite field of characteristic $2$, $\lambda\mathfrak{H}(\mathbb{K})\subseteq\mathfrak{H}(\mathbb{K})$ only if $\lambda\in\{0,1\}$. Otherwise $\lambda\mathfrak{H}(\mathbb{K})\cap\mathfrak{H}(\mathbb{K})$ is an index $4$ subgroup of $\mathbb{K}^+$, and $\lambda\mathfrak{H}(\mathbb{K})$ and $\mathfrak{H}(\mathbb{K})$ generate $\mathbb{K}^+$ as an additive module.
\end{lemma}

\begin{proof}
Let the number of elements in $\mathbb{K}$ be $2^n$. Then it can be checked that $a\in\mathfrak{H}(\mathbb{K})$ if and only if $p(a)=0$ for $p(x)=\sum_{i=0}^{n-1}x^{2^i}$. Therefore $\mathfrak{H}(\mathbb{K})$ is in fact the set of roots of $p(x)$, all of which are simple roots.

Assume now that $\lambda\ne 0$.
Since $\mathfrak{H}(\mathbb{K})$ and $\lambda\mathfrak{H}(\mathbb{K})$ are both index $2$ subgroups of $\mathbb{K}^+$, either the two are the same, or their intersection is in fact an index $4$ subgroup. However, if they are the same, that means that the roots of $p(x)$ and $p(\lambda x)$ are identical, which is only possible if $\lambda=1$.
\end{proof}

\section{Universal conformal geometry}

Let us recall a few key definitions from \cite{juhasz2016univ}. We will assume $\mathbb{K}$ to be a perfect field of characteristic $2$ and the dimension of the geometry to be odd.

\begin{definition}
A \textbf{universal conformal geometry} of dimension $2\nmid n$ is a tuple $(V,Q,P,L)$ where $V$ is a vector space of dimension $n+3$ over $\mathbb{K}$, $Q$ is a non-degenerate quadratic form with non-degenerate associated bilinear form $B$, and $P$, $L\in\mathbb{P}V$ are orthogonal.
\end{definition}

We shall extend this definition such that $V$ may be a non-degenerate virtual quadratic space. This is only relevant when the dimension of the geometry is even, in which case the virtual quadratic space is of the form $(V,\Omega^\perp)$ for some $\Omega\in V$.

\begin{definition}
A \textbf{universal conformal geometry} of dimension $2\mid n$ is a tuple $(V,Q,\Omega,P,L)$ where $V$ is a vector space of dimension $n+4$ over $\mathbb{K}$, $\Omega\in\mathbb{P}V$ some vector up to scalar, $(V,Q,\Omega^\perp)$ a non-degenerate virtual quadratic space with $\Rad Q|_{\Omega^\perp}\ne 0$, with $P$, $L\in\mathbb{P}V$ are orthogonal.
\end{definition}

We shall refer to the condition $P\perp L$ as assumption (0).
We will also assume that $L$, $P$, $\Omega$ are all linearly independent, and denote this condition by (1). Furthermore we shall assume that either $Q(P)\ne 0$ or $\Arf\langle\Omega, P\rangle\ne\infty$, and the same for $L$ instead of $P$. This later condition shall be denoted by (2).

Some further definitions:

\begin{definition}
The set $[[Q]]$ of points in $\mathbb{P}V$ where $Q$ is zero is called the \textbf{Lie quadric}. The geometry is \textbf{empty} if $[[Q]]$ is empty, and \textbf{non-empty} otherwise.
\end{definition}

\begin{definition}
A \textbf{hypercycle} is an element $c\in\mathbb{P}V$ such that $Q(c)=0$. A hypercycle $c$ is a \textbf{pointcycle} or \textbf{point} if $B(P,c)=0$, and a \textbf{hyperplane} or \textbf{hyperplanecycle} if $B(L,c)=0$. A hypercycle that is both a point and a hyperplane is \textbf{ideal} (or \textbf{ideal hyperplane}). When the dimension is two, a hyperplane is called a \textbf{line}.
\end{definition}

\begin{definition}
Two hypercycles $c_1$, $c_2$ are \textbf{touching} or \textbf{incident} if $B(c_1,c_2)=0$.
\end{definition}

\begin{definition}
An \textbf{isometry} of the geometry is an isometry of $(V,Q)$ that preserves $P$, $L$ and $\Omega$.
An \textbf{isometry of a hypercycle} $c$ is an isometry of the geometry that preserves $c$.

Two geometries $(V,Q,\Omega,P,L)$ and $(V',Q',\Omega',P',L')$ are \textbf{isometric} if there is an orthogonal map $V\to V'$ under which the image of $\Omega$, $P$, $L$ is $\Omega'$, $P'$, $L'$ up to scalar multiple.
\end{definition}

We will say for a class of geometries that the geometry is \emph{defined} if all the geometries in the class are isometric.

\section{Geometry in characteristic $2$}

Once again, $\mathbb{K}$ will denote a perfect field of characteristic $2$, such as a finite field of $2^k$ elements, and we assume that $(V,Q,\Omega,P,L)$ is a universal conformal geometry.

\begin{lemma}
A quadratic space $V_0$ of dimension $3$ can always be embedded into a non-degenerate quadratic space of dimension at least $6$, and if the associated bilinear form is not constant $0$ on $V_0$, it can be embedded into a quadratic space of dimension $4$ as well. Furthermore, the Arf invariant of the embedding space can be arbitrary.
\end{lemma}

\begin{proof}
First we shall prove that such an embedding exists for dimension $6$.
Let us fix a basis of $V_0$, $e_1$, $e_2$, $e_3$, and let $V$ be direct sum of $V_0$ and $U$, where $U$ is generated by the vectors $f_1$, $f_2$, $f_3$. Let the quadratic form extend to $V$ through $B(e_i,f_i)=1$, and any other $B(e_i,f_j)=B(f_i,f_j)=Q(f_i)=0$ for $i\ne j$. This is a non-degenerate quadratic space, since for any non-zero vector of the form $v+u$ where $v\in V_0$ and $u\in U$, $B(v+u,f_i)\ne 0$ for some $i\in\{1,2,3\}$ if $v\ne 0$, otherwise $B(v+u,e_i)\ne 0$ since $u\ne 0$.

Assume now that the bilinear form is not constant $0$. This means that there are two vectors $e_1$ and $e_2$ such that $B(e_1,e_2)\ne 0$, and there is a linearly independent vector $e_3$ orthogonal to both of those. Then we may extend $V_0$ by a new vector $e_4$ such that $B(e_3,e_4)=1$, and the resulting quadratic space is non-degenerate. The Arf invariant of the whole space is $\Arf\langle e_1,e_2\rangle+\Arf\langle e_3,e_4\rangle$.

Finally, we shall prove that the Arf invariant can be set to be some arbitrary value $a$. Let us choose an embedding of $V_0$ into $V$ with Arf invariant $a_0$, and two vectors $e$ and $f$ in $V$ such that $a_1=\Arf\langle e,f\rangle\not\in\{0,\infty\}$. Such a pair of vectors always exists if $\dim V>2$, since we may choose an $e$ such that $Q(e)\ne 0$, an $f$ such that $B(e,f)\ne 0$, and if $Q(f)=0$, we may replace $f$ by $f+f'$ where $f'\in\langle e,f\rangle^\perp$ with $Q(f')\ne0$. Furthermore, we may assume that $e\in V_0$ and $f\not\in V_0$, the later because if by the previous construction $f$ happens to be in $V_0$, we may choose some vector $f'\in\langle e,f\rangle^\perp$ such that $Q(f')\ne Q(f)$, and replace $f$ with $f+f'$.

We may replace $Q$ with another $Q'$ if we specify the values of the quadratic form and associated bilinear form on elements of a basis. Therefore let us extend $e$, $f$ into a basis of $V$, and replace $Q$ with $Q'$ in a way such that the only difference is $Q'(f)=\frac{a+a_0+a_1}{a_1}Q(f)$, which results in the expected Arf invariant.
\end{proof}

\begin{theorem}
Assume that $\dim V\ge 6$ is fixed. Then the value $\Arf V$, defined up to $\mathfrak{H}{\mathbb{K}}$, and the values $\Arf\langle \Omega,P\rangle$ and $\Arf\langle \Omega,L\rangle$ define the geometry (this also holds if $\dim V=4$ unless $\Arf\langle\Omega,P\rangle=\Arf\langle\Omega,L\rangle$ are both zero). We may also assume that $Q(\Omega)=1$.
\end{theorem}

\begin{proof}
Since $\Rad Q|_{\Omega^\perp}\ne 0$, $Q(\Omega)\ne 0$. In a perfect field, we may scale $\Omega$ in a way such that $Q(\Omega)=1$, and this does not change the parameters $\Arf V$, $\Arf\langle\Omega,P\rangle$ and $\Arf\langle\Omega,L\rangle$.

The subspace $V_0:=\langle\Omega,P,L\rangle$, which is of dimension $3$ by assumption (1), and the restriction of the quadratic form is entirely determined by the value of the quadratic form and bilinear form on $\Omega$, $P$, $L$ and the pairs, respectively. By assumption (0), $B(P,L)=0$. If $\Arf\langle\Omega,P\rangle\ne\infty$, we may rescale $P$ so that $B(\Omega,P)=1$. Otherwise by assumption (2), we may rescale $P$ so that $Q(P)=1$, since in a perfect field of characteristic $2$, $\lambda^2=Q(P)$ has a single solution.

Given two geometries $V$ and $V'$ with identical Arf invariants, we identify the quadratic spaces $(V,Q)$ and $(V',Q')$. Then the two generated subspaces $V_0$ and $V_0'$ are isometric, and this isometry extends into an isometry of $V$.

Whenever $\Arf\langle\Omega,P\rangle$ and $\Arf\langle\Omega,L\rangle$ are fixed, this gives us a quadratic space $V_0$, and by the previous lemma, we may always embed it into a quadratic space of the desired invariant. Hence such a geometry always exists.
\end{proof}

We shall denote by $\Arf(x):=\Arf\langle\Omega,x\rangle$.
From now on, we will concentrate on geometries of dimension $2$, hence $\dim V=6$.

\begin{theorem}
Let us denote by $\Sigma$ a hyperbolic space of dimension $2$.
A non-degenerate geometry of dimension $2$ is always isomorphic to one of the form $(\Sigma\oplus\Sigma\oplus\Sigma,\Omega)$.
\end{theorem}

\begin{proof}
Since it is non-degenerate, we have a line $\ell\perp L$ and a point $p\perp P$ such that the point is on the line: $p\perp\ell$, and neither of them are ideal, hence $\ell\not\perp P$, $p\not\perp L$. Therefore $\langle\ell,P\rangle$ and $\langle p,L\rangle$ are orthogonal hyperbolic spaces, and since the Arf invariant of the space is $0$, $V$ is isomorphic to $\Sigma\oplus\Sigma\oplus\Sigma$.
\end{proof}

Let us assume that the geometry is of the above form.
From a geometric point of view, fixing $\Arf(P)$ and $\Arf(L)$ will distinguish between similar geometries. Note that the isometry group of a geometry depends only on the choice of $P$, $L$, and the subspace $\langle P,L,\Omega\rangle$, and that none of the objects defined thus far depend on the choice of $\Omega$.

\begin{definition}
Two geometries $(V,Q,P,L,\Omega)$ and $(V',Q',P',L',\Omega')$ are \textbf{transformation equivalent} if there is an isometry between them that sends $P$ to $P'$, $L$ to $L'$ (up to scalar), and the subspace $\langle P,L,\Omega\rangle$ to $\langle P',L',\Omega'\rangle$.
\end{definition}

We get transformation equivalent spaces by replacing $\Omega$ with some $\Omega'=\Omega+\alpha P+\beta L$ for some $\alpha$, $\beta\in\mathbb{K}$, assuming that $Q(\Omega')\ne 0$.

\begin{lemma}
Let us replace $\Omega$ by $\Omega+\alpha P+\beta L$ for some $\alpha$, $\beta\in\mathbb{K}$, assuming that the new $Q(\Omega)\ne 0$. If either $\Arf(P)$ or $\Arf(L)$ are $0$ or $\infty$, both of their values shall remain the same. Otherwise, $\Arf(L)$ with the new $\Omega$ will become
$$\Arf(L)+{B(\Omega,P)^2Q(L)\over Q(P)B(\Omega,L)^2}\mathfrak{H}\left(\alpha{Q(P)\over B(\Omega,P)}\right)+\mathfrak{H}\left(\beta{Q(L)\over B(\Omega,L)}\right)$$
and similarly for $\Arf(P)$, only the occurences of $\alpha$, $P$ and $\beta$, $L$ exchanged.
\end{lemma}

\begin{proof}
This is a simple calculation, using the definition of $\Arf$ and the fact that $B(\Omega',P)=B(\Omega,P)$.
\end{proof}

\begin{corollary}
A class of transformation equivalent geometries, and thus the group of isometries of a geometry is identified by $\rho:={\Arf(L)\over\Arf(P)}$, by $\Arf(L)$ and $\Arf(P)$ if either of them is $0$ or $\infty$, and in case of $\rho=1$, by $\Arf(L)+\mathfrak{H}(\mathbb{K})$.
\end{corollary}

\begin{proof}
Two geometries are related by the transformation established in the previous lemma.
By the previous lemma, the statement of the corollary is clear if $\Arf(P)$ or $\Arf(L)$ is $0$ or $\infty$, otherwise $\rho$ is preserved by the above transformation. Then it is enough to now $\rho$ and $\Arf(L)$ to reconstruct the geometry.

Replacing $\alpha{Q(P)\over B(\Omega,P)}$ and $\beta{Q(L)\over B(\Omega,L)}$ by $x$ and $y$, respectively, we get identical geometries with $\Arf(L)$ changing to $\Arf(L)+\rho^{-1}\mathfrak{H}(x)+\mathfrak{H}(y)$, assuming that $\Arf(L)\not\in\{0,\infty\}$. If $\rho\mathfrak{H}(\mathbb{K})\ne\mathfrak{H}(\mathbb{K})$, then $\rho^{-1}\mathfrak{H}(\mathbb{K})$ and $\mathfrak{H}(\mathbb{K})$ are non-identical index $2$ subgroups of $\mathbb{K}^+$, hence they generate $\mathbb{K}$, and $\Arf(L)$ can take any value, but resulting in equivalent geometries. Otherwise, $\Arf(L)$ is defined up to $\mathfrak{H}(\mathbb{K})$, which is only possible if $\rho=1$ by \ref{Lindex}.
\end{proof}

\section{Distance in characteristic $2$}

In \cite{juhasz2016univ}, distance was defined through fixing a line, and considering the isometry group that fixes the line. In dimension $2$, a line is given through a single hypercycle $\ell$ orthogonal to $L$. First, we need to consider the case when $\ell$ is linearly dependent from $\Omega$, $L$, and $P$.

\begin{definition}
The \textbf{group of isometries of a cycle} $c$ (respectively, that of a line $\ell$) is the set of isometries of $V$ that fix $\Omega$, $P$, $L$ and $c$ (respectively, $\ell$). We call a cycle $c$ (respectively, a line $\ell$) \textbf{independent}, if it is linearly independent from $\Omega$, $P$, $L$ as vectors in $V$.
\end{definition}

\begin{lemma}
A non-ideal real line $\ell$ is independent, unless $B(\Omega,L)=0$ and $\ell=Q(\Omega)^{1/2}L+Q(L)^{1/2}\Omega$ or its scalar multiple.
\end{lemma}

\begin{proof}
By (0), the first three are definitely linearly independent. We know that $Q(\ell)=0$ since it is a cycle, $B(L,\ell)=0$ since $\ell$ is a line, $B(P,\ell)\ne0$ since it is non-ideal, and that $B(\Omega,\ell)=0$ since it is real. Assuming that $\ell$ is in the space generated by $P$, $\Omega$, $L$, it can be written as a linear combination $\alpha P+\beta L+\gamma\Omega$. Then we may check all these conditions, among which $0\ne B(P,\ell)=\gamma B(\Omega,P)$ means that $\gamma\ne 0$, and $0=B(L,\ell)=\gamma B(\Omega,L)$ means that $B(\Omega,L)=0$. The details are simple but technical, and are left to the reader.
\end{proof}

\begin{lemma}
The group of isometries of an independent line $\ell$ is either isomorphic to the group of isometries of a dimension $2$ quadratic space, or to $\mathbb{K}^+\times\mathbb{F}_2^+$.
\end{lemma}

\begin{proof}
Let us consider the subspace $V_0=\langle\Omega,P,L,\ell\rangle$ that is of dimension $4$ by condition (1). If the restricted bilinear form $B|_{V_0}$ is non-degenerate, then the group of isometries of the line is isomorphic to the group of isometries of the subspace $V_0^\perp$, which is of dimension $2$, and is also a non-degenerate quadratic space.

If $B|_{V_0}$ is degenerate, we may show that $\dim\Ker B|_{V_0}=2$. Consider in general a non-degenerate quadratic space $V$ with a subspace $V_0$ and a non-zero vector $v\in\Ker B|_{V_0}$. Then we may find a $u\in V$ such that $B(u,v)=1$, since the quadratic space $V$ is non-degenerate. Then we may choose a subspace $V_0'$ of $V_0$ of codimension $1$ that is orthogonal to $\langle u,v\rangle$, embedded inside $\langle u,v\rangle^\perp$, a non-degenerate quadratic subspace $V'$ of $V$ of dimension $2$ less than $V$. Clearly, if $\dim V_0$ is even, then $\dim V_0'$ is odd, hence it is a degenerate space, and the original $\dim\Ker B|_{V_0}>1$. On the other hand, repeating this method recursively, we can see that $\dim\Ker B|_{V_0}\le \dim V-\dim V_0$, which in our initial case was $2$.

So we may assume that $\dim\Ker B|_{V_0}=2$ and choose $e_1$, $e_2\in V_0$ and $e^1$, $e^2\in V$ such that $B(e_i,e^i)=1$, and every other pair evaluates to $0$ under the bilinear form. Since $V_0$, $e^1$ and $e^2$ generate the whole space $V$, an isometry is determined by the image of $e^1$, $e^2$. Furthermore, denoting $U:=\langle e_1,e_2,e^1,e^2\rangle$, since $U^\perp\subseteq V_0$, their images are linear combinations of these four vectors.

Let us fix an isometry $\varphi$ of $U$ that fixes $e_i$ and $e_j$. Since $B(e_i,e^j)=B(e_i,\varphi(e^j))$, we have $\varphi(e^i)=e^i+\alpha_i e_i+\beta_i e_{i'}$ where $i'$ denotes the other index in the set $\{1,2\}$. Since $B(e^1,e^2)=B(\varphi(e^1),\varphi(e^2))$, this means that $\beta_1^2=\beta_2^2$, and since we're in characteristic $2$, $\beta=\beta_1=\beta_2$. Finally, $Q(e^i)=Q(\varphi(e^i))$, hence $\beta^2=\frac{\alpha_i^2Q(e_i)+\alpha_i}{Q(e_{i'})}$ for $i\in\{1,2\}$. This means that $\frac{\alpha_1^2Q(e_1)+\alpha_1}{Q(e_2)}=\frac{\alpha_2^2Q(e_2)+\alpha_2}{Q(e_1)}$, which when multiplied by $Q(e_1)Q(e_2)$ gives $\mathfrak{H}(\alpha_1Q(e_1))=\mathfrak{H}(\alpha_2Q(e_2))$, hence $\alpha_1Q(e_1)+\alpha_2Q(e_2)\in\{0,1\}$. By referring to this value as $\varepsilon$, the isometry is completely determined by the pair $(\alpha_1,\varepsilon)\in\mathbb{K}^+\times\mathbb{F}_2^+$, since $\beta$ may be chosen uniquely due to the field being perfect.

It is easy to see, for instance by looking at the image of $e_1$ under isometries, that the map $\varphi\to(\alpha_1,\varepsilon)$ is in fact a group isomorphism.
\end{proof}

\begin{lemma}
The group of translations via an independent, non-ideal line $\ell\perp L$ depends only on the values $\rho:=\frac{\Arf(L)}{\Arf(P)}$, $\Arf\langle\Omega,L\rangle$ given up to $\mathfrak{H}(\mathbb{K})$, and if $\rho=\infty$ on $\frac{Q(L)Q(P)}{B(\Omega,L)^2}B(\Omega,\ell)^2$, otherwise on $\frac{Q(P)}{B(\Omega,P)}B(\Omega,\ell)$. Furthermore, if $B(\Omega,\ell)=0$, it only depends on $\Arf(L)$ given up to $\mathfrak{H}(\mathbb{K})$.
\end{lemma}

\begin{proof}
Since $P$, $L$, $\Omega$ and $\ell$ are fixed under such an isometry, the isometry group is completely determined by $\Arf\langle\ell,P,L,\Omega\rangle^\perp$ up to $\mathfrak{H}(\mathbb{K})$.

Let us consider an element $\ell\perp L$, $Q(\ell)=0$, $\ell\not\perp P$. Let us assume $B(P,\ell)=1$. Then $\langle \ell,P\rangle$ is a hyperbolic space, and hence $\Arf\langle\ell,P\rangle^\perp=0$. To extend $\{\ell, P\}$ into a basis, we need to project $\Omega$ and $L$ onto $\langle\ell,P\rangle^\perp$, the later of which is already inside due to assumption (0). Then $\Omega'=\Omega+B(\Omega,\ell)P+B(\Omega,P)\ell$ is the projection of $\Omega$.

Since $\Arf\langle\ell,P\rangle^\perp=0$, $\Arf\langle\ell,P,L,\Omega\rangle^\perp=\Arf\langle L,\Omega'\rangle$, which is
$${Q(L)Q(\Omega)\over B(\Omega,L)^2}+{Q(L)B(\Omega,P)\over B(\Omega,L)^2}B(\Omega,\ell)+{Q(L)Q(P)\over B(\Omega,L)^2}B(\Omega,\ell)^2$$
which is equal to
$$\Arf\langle\Omega,L\rangle+{Q(L)Q(P)\over B(\Omega,L)^2}B(\Omega,\ell)^2$$
if $\Arf\langle\Omega,P\rangle=0$, otherwise to
$${Q(L)Q(\Omega)\over B(\Omega,L)^2}
+{Q(L)B(\Omega,P)^2\over B(\Omega,L)^2Q(P)} \mathfrak{H}\left({Q(P)\over B(\Omega,P)}B(\Omega,\ell)\right)$$
or equivalently
$$\Arf\langle\Omega,L\rangle+{\Arf\langle\Omega,L\rangle\over\Arf\langle\Omega,P\rangle}\mathfrak{H}\left({Q(P)\over B(\Omega,P)}B(\Omega,\ell)\right).$$
The isometry group depends on this value, up to $\mathfrak{H}(\mathbb{K})$. When $B(\Omega,\ell)=0$, this is simply $\Arf\langle\Omega,L\rangle$.
\end{proof}

The previous lemma suggests that we look at elements that are orthogonal to $\Omega$ separately.

\begin{definition}
A \textbf{real hypercycle} (or \textbf{real point}) is a hypercycle (or point) $c\in V$ such that $c\perp\Omega$. A \textbf{virtual hypercycle} (or \textbf{virtual point}) is a hypercycle (or point) that is not real.\footnote{These definitions are not entirely compatible with those defined in \cite{juhasz2016univ}. However, they fulfill similar roles, pertaining to the existence of solutions of quadratic polynomials.}
\end{definition}

Note that this is the first definition that uses $\Omega$ directly, and thus the choice of real and virtual hypercycles distinguishes between otherwise transformation equivalent spaces.

\begin{lemma}
Given a cycle $c\in V$,
the group of isometries that fix this cycle acts transitively on those non-ideal points $p$ where ${B(\Omega,p)\over B(L,p)}$ is a fixed value, and $p$ is linearly independent from $\Omega$, $P$, $L$, $c$.
\end{lemma}

\begin{proof}
For a non-ideal point $p$, we have $B(L,p)\ne 0$, and by rescaling we may assume that $B(L,p)=1$. Note that this does not change the ratio ${B(\Omega,p)\over B(L,p)}$ which is fixed up to scalar multiples. In fact, if $B(L,p)=1$, it becomes equal to $B(\Omega,p_1)=B(\Omega,p_2)$.

Assume given two non-ideal points $p_1$ and $p_2$ on the cycle $c$ with the above condition and scaling, thus having $B(\Omega,p_1)=B(\Omega,p_2)$. Then the subspaces $\langle\Omega,P,L,c,p_i\rangle$ are isometric for $i\in\{1,2\}$. Therefore there is an isometry between these two subspaces that sends $p_1$ to $p_2$ and fixes the others, and this extends into an isometry of $V$.
\end{proof}

The above lemma shows that just as for \cite{juhasz2016univ}, it is possible to define translations through lines for most pairs of points. To define distances, we need to talk about orientation, which is generally not well-defined in characteristic $2$, since $1=-1$. However, the arising groups have some very useful properties.

\begin{definition}
Let us denote by $\Ort(\alpha)$ the group of isometries of a quadratic space of dimension $2$ whose Arf invariant is $\alpha$, and let $\Ort(\infty)$ be $\mathbb{K}^+\times\mathbb{F}_2^+$.
\end{definition}

In particular, $\Ort(0)$ is the group of isometries of the hyperbolic space.

\begin{lemma}
For any $\alpha\in\mathbb{K}\cup\{\infty\}$, there is an index $2$ subgroup $\Ort(\alpha)^+$ of $\Ort(\alpha)$ that acts transitively on the set of points of a certain non-zero norm.
\end{lemma}

\begin{proof}
When $\alpha=\infty$, $\Ort(\alpha)$ is isomorphic to $\mathbb{K}^+\times\mathbb{F}_2^+$, in which there is clearly a unique index $2$ subgroup, isomorphic to $\mathbb{K}^+$ (it is unique since $\mathbb{K}^+$ is of odd order).

Let us fix a (not necessarily symmetric) bilinear form $A$ such that $Q(x)=A(x,x)$. Such a form always exists. Let $M$ be an isometry, that is, $Q(M(x))=Q(x)$. Then $A(M(x),M(y))=A(x,y)+B_M(x,y)$ for some anti-symmetric bilinear form $B_M$, i. e. $B_M(x,x)=0$. In characteristic $2$, every anti-symmetric bilinear form is a symmetric bilinear form as well.

Since the dimension of the space is $2$, the dimension of the space of symmetric bilinear forms is $1$, hence $B_M=\lambda_MB$ for some $\lambda_M\in\mathbb{K}$ dependent on $M$, where $B$ is the associated bilinear form to $Q$.

Written in matrix form, we get $M^TAM=A+\lambda_MB$. Let us assume a basis where $B=\left[\begin{matrix}0& 1\cr 1& 0\end{matrix}\right]$. Taking the determinant of both sides, we get $\det M^2\det A$ on the left, and $\det A+\lambda_M\cotr A+\lambda_M^2$ where $\cotr A$ is the sum of the elements in the co-diagonal. Since $M^TBM=B$ and $\det B\ne 0$, we get $\det M^2=1$. Since $B(x,y)=A(x,y)+A(y,x)$, $\cotr A=1$. Therefore we get $\lambda_M+\lambda_M^2=1$, whence $\lambda_M\in\{0,1\}$.

Let us denote by $\Ort(\alpha)^+$ the set of matrices $M$ where $\lambda_M=0$. It is fairly easy to see that this is in fact a subgroup of index $2$.
\end{proof}

Therefore we may define the distance in the following manner.

\begin{definition}
Given two non-ideal real points $p_1$ and $p_2$ on a non-ideal real line $\ell$, the \textbf{oriented distance} of $p_1$ and $p_2$ is the unique element of the group of isometries fixing $\ell$, a group isomorphic to $\Ort(\Arf\langle\Omega,L\rangle)$, that sends $p_1$ to $p_2$. The oriented distance is of $p_2$ and $p_1$ is the group theoretic inverse of the oriented distance of $p_1$ and $p_2$. The \textbf{distance} of two points $p_1$ and $p_2$ is then the equivalence class of the element quotiented by the group inversion, in $\Ort(\Arf\langle\Omega,L\rangle)/\{\gamma\sim\gamma^{-1}\}$.
\end{definition}

In a finite field of characteristic $2$, the group $\mathbb{K}^+/\mathfrak{H}(\mathbb{K})$ contains two elements, represented by $0$ and $\mathbf{e}$.
Therefore, apart from $\Ort(\infty)$, there are only two orthogonal groups: $\Ort(0)$, the isometries of the hyperbolic space, and $\Ort(\mathbf{e})$. Note that the quadratic space corresponding to $\Ort(\mathbf{e})$, with quadratic form $x^2+xy+\mathbf{e}y^2$, is similar to the elliptic case in other finite fields, in particular that all orbits, apart from $\{0\}$, contain $|\mathbb{K}|+1$ points.

We get the following classification for uniform geometries in finite fields of characteristic $2$, using the notation $\Arf(x):={Q(x)Q(\Omega)\over B(x,\Omega)^2}$:

\begin{center}
\begin{tabular}{|c|c|c|c|}
\hline
\diagbox{$\Arf(P)$}{$\Arf(L)$}&	$\mathbf{e}$&	$\infty$&	$0$\\
\hline
$\mathbf{e}$&	elliptic&			parabolic&		hyperbolic\\
\hline
$\infty$&	dual parabolic&			Laguerre/Galilei&	dual Minkowski\\
\hline
$0$&		dual hyperbolic&		Minkowski&		anti-de Sitter\\
\hline
\end{tabular}
\end{center}

\medskip

Note the similarity to finite fields of other characteristics, where $\mathbf{e}$ would denote a quadratic non-residue (see \cite{juhasz2016univ}):

\begin{center}
\begin{tabular}{|c|c|c|c|}
\hline
\diagbox{$Q(P)$}{$Q(L)$}&	$\mathbf{e}$&	$0$&	$1$\\
\hline
$\mathbf{e}$&	elliptic&			parabolic&		hyperbolic\\
\hline
$0$&		dual parabolic&			Laguerre/Galilei&	dual Minkowski\\
\hline
$1$&		dual hyperbolic&		Minkowski&		anti-de Sitter\\
\hline
\end{tabular}
\end{center}

Note that unlike in the characteristic non-$2$ case, the values $\Arf(P)$ and $\Arf(L)$ given up to $\mathfrak{H}(\mathbb{K})$ do not define the geometry completely.

\section*{Acknowledgement}

The author's research was supported by the grant EFOP-3.6.1-16-2016-00001 (``Complex improvement of research capacities and services at Eszterhazy Karoly University'').

I would like to thank \'Akos G. Horv\'ath for encouraging me to write this article, and who provided me with advice in the preparation of this article.


\begin{thebibliography}{XXX}

\bibitem{baeza1985}
Ricardo Baeza, Remo Moresi: On the Witt-Equivalence of Fields of Characteristic $2$. \emph{Journal of Algebra.} \textbf{92.} pp446--453. 1985.

\bibitem{hestens2000}
David Hestens, Hongbo Li, Alyn Rockwood. \emph{A unified algebraic framework for classical geometry}. 2000.

\bibitem{juhasz2016univ}
M\'at\'e L. Juh\'asz: A universal linear algebraic model for conformal geometries. 2016. \texttt{arXiv:1603.06863}.

\bibitem{juhasz2016isom}
M\'at\'e L. Juh\'asz: Isometries of virtual quadratic spaces. \emph{Annales Univ. Sci. Sect. Math.}, \textbf{59}. 2016. 123--132.

\bibitem{kitaoka1993}
Yoshiyuki Kitaoka. \emph{Arithmetic of quadratic forms}. Cambridge University Press. 1993.

\bibitem{yaglom1969}
I. M. Yaglom. \emph{Galilei's Relativity Principle and Non-Euclidean Geometry}. Nauka. 1969.

\end{thebibliography}
\end{document}